\newif\ifarxiv\arxivtrue
\ifarxiv
\documentclass[11pt]{article}
\else
\documentclass[12pt]{article}
\fi
\usepackage[margin=1in]{geometry}
\usepackage{amsmath,amssymb,amsthm,booktabs}
\usepackage[hidelinks,colorlinks=false]{hyperref}
\newtheorem{theorem}{Theorem}[section]
\newtheorem{lemma}[theorem]{Lemma}
\newtheorem{proposition}[theorem]{Proposition}
\newtheorem{corollary}[theorem]{Corollary}
\newtheorem{conjecture}[theorem]{Conjecture}
\theoremstyle{definition}
\newtheorem{definition}[theorem]{Definition}
\theoremstyle{remark}
\newtheorem{remark}[theorem]{Remark}
\newcommand{\pmin}{p_{\min}}
\newcommand{\fS}{\mathfrak{S}}
\newcommand{\NEW}[1]{\ifarxiv\ (#1)\fi}
\title{Matrix $B$: Exact Prime Counting, Algebraic Row Correspondences,\\ and Semiprime Distributions}
\author{Wujie Shi\\[2pt]
{\small Department of Mathematics, Chongqing University of Arts and Sciences,}\\
{\small Chongqing 402160, P.R. China;}\\
{\small School of Mathematical Sciences, Suzhou University, Suzhou 215006, P.R. China}\\
{\small \texttt{shiwujie@outlook.com}}\ifarxiv\\[4pt]{\small arXiv:2605.21529v6 (replaces v5)}\fi}
\date{}
\begin{document}
\maketitle


\begin{abstract}
Matrix $B=(b_{ij})$ with $b_{ij}=(2j+1)(2j+2i-1)$ was introduced in \cite{Shi} as an additive sieve for odd primes. In this paper we develop both analytic and algebraic approaches to matrix $B$. On the analytic side, we prove the exact prime counting formula $P_k=N_k-S_k+E_k$ for the intervals $I_k=((2k-1)^2,(2k+1)^2)$, yielding the equivalent combinatorial condition $P_k\ge 1\iff E_k>S_k-N_k$. We verify $P_k\ge1$ for all $k\le 999{,}999{,}999$ using the prime-gap tables of Oliveira e Silva, Herzog and Pardi, and for $k_0\le k\le 5.49\times10^{11}$ via the Baker--Harman--Pintz theorem, where $k_0$ denotes the inexplicit threshold in that theorem. On the algebraic side, we establish a natural injection $\varphi_i$ from the $(i+1)$-th row into the $i$-th row and clarify its structural content; we prove a closed-form evaluation of the row Dirichlet series at $s=1$; and we describe the transposition symmetry of $B$, under which the rows encode differences of primes and the anti-diagonals encode sums of primes. For semiprime distributions, Row~1 trivially contains infinitely many semiprimes; for $i\ge2$ the question is equivalent to Polignac's conjecture for the gap $2(i-1)$; by Chen's theorem every row contains infinitely many entries with at most three prime factors; and by the Maynard--Tao theorem with the Polymath refinement at least one row $i_0\le124$ contains infinitely many semiprimes.

\ifarxiv\smallskip\noindent\textbf{v4--v5.} \fi
We introduce the minimal-anchor function $\pmin(d)$, the least odd prime $p$ such that $p+d$ is prime (sequence A020483 of the OEIS at index $d/2$), whose finiteness for all even $d$ is exactly the weak Polignac (Maillet) conjecture, i.e.\ the statement that every row of $B$ contains a semiprime. For each fixed $z$ the set $\{d \text{ even}:\pmin(d)\le z\}$ has density zero, with the asymptotic $(\pi(z)-1)X/\log X$, whereas density-one coverage of the rows holds by a classical theorem of Lavrik. We prove unconditional lower bounds: for every $\psi\to0$, $\pmin(d)>\psi(d)\log d\log\log d$ for almost all even $d$, which is the conjectured typical order, and $\max_{d\le X}\pmin(d)\ge(\tfrac12+o(1))\log X\log\log X$. The same counting gives the corresponding lower bounds for the least prime in a Goldbach partition (OEIS A002373).

\ifarxiv\smallskip\noindent\textbf{New in v6.} \fi
We restate the consequences of the bounded-gap theorems for the rows of $B$ in a form that weakens every-row coverage non-trivially: for every even $d$ at least one of the rows of gaps $d,2d,\dots,123d$ contains infinitely many semiprimes, such rows have positive lower density, and (after Pintz) they occur in every block of $C$ consecutive rows. Finally we examine how far the structure of $B$ alone can go towards every-row coverage: the row-transfer identity, the transposition symmetry and the congruence obstructions (``fences'') are all shown to be insensitive to the emptiness of a single row, by means of explicit prime-like deletion sets in which exactly one row is empty. Every-row coverage (= weak Polignac) is explicitly left open.
\end{abstract}

\section{Introduction}
\subsection{Background and motivation}
In \cite{Shi} the author introduced a new addition-based sieve for prime numbers, centered on the infinite matrix $B=(b_{ij})$,
\[ b_{ij}=(2j+1)(2j+2i-1)=4j^2+4ij+2i-1,\qquad i,j=1,2,3,\dots \]
Every odd composite appears in $B$ at least once; an odd integer $\ge3$ is prime if and only if it is absent from $B$. While the original approach in \cite{Shi} was analytic-combinatorial, this paper develops a systematic algebraic theory of matrix $B$. The rows of $B$ carry a natural injection structure, the Dirichlet series of each row admits a clean closed-form evaluation at $s=1$, and the matrix has a transposition symmetry. We also clarify the precise relationship between semiprime distributions in the rows and classical prime-gap conjectures.

\subsection{Main contributions}
\begin{enumerate}\setlength\itemsep{1pt}
\item Exact prime counting formula (Theorem~\ref{exact}): $P_k=N_k-S_k+E_k$, with equivalent condition $P_k\ge1\iff E_k>S_k-N_k$.
\item Partial unconditional results (Theorems~\ref{gap}--\ref{bhp}): $P_k\ge1$ for all $k\le999{,}999{,}999$ unconditionally, and for $k_0\le k\le5.49\times10^{11}$ via Baker--Harman--Pintz.
\item Row injection (Lemma~\ref{inj}): a natural injection $\varphi_i$ from Row $i+1$ into Row $i$; adjacent-row entries share a common factor (Proposition~\ref{shared}).
\item Row value identity (Theorem~\ref{rowvalue}): for $i\ge2$, $F_i(1)=\frac{1}{2(i-1)}\sum_{k=1}^{i-1}\frac{1}{2k+1}$.
\item Semiprime distributions (Theorems~\ref{row1}--\ref{mtp}): Row~1 has infinitely many semiprimes (trivial); for $i\ge2$ equivalent to Polignac for gap $2(i-1)$; Chen gives near-semiprimes in every row; Maynard--Tao--Polymath gives some row $i_0\le124$ with infinitely many semiprimes.
\item The minimal-anchor function $\pmin$ (Section~\ref{sec65}): for each fixed $z$ the covered set $\{d:\pmin(d)\le z\}$ has density zero with asymptotic $(\pi(z)-1)X/\log X$ (Theorem~\ref{fixed}); almost every row contains the Hardy--Littlewood-predicted number of semiprimes, by Lavrik's theorem restated in the matrix-$B$ framework (Theorem~\ref{lavrik}); quantitative conjectures on $\pmin$ (Conjecture~\ref{conj}) with numerical data (Table~\ref{tab2}).
\item Uniform union bound (Theorem~\ref{union}) and unconditional lower bounds for $\pmin$ (Theorems~\ref{almostall}, \ref{extremal}), with the Goldbach analogue (Theorem~\ref{goldbach}); anchor-by-anchor asymptotics (Proposition~\ref{anchorby}) and Table~\ref{tab3}.
\item\NEW{new in v6} Weakened forms of every-row coverage (Section~\ref{sec68}): for every even $d$ one of the rows of gaps $d,2d,\dots,123d$ contains infinitely many semiprimes (Theorem~\ref{multiples}); such rows have positive lower density (Corollary~\ref{posdens}) and occur in every block of $C$ consecutive rows (Theorem~\ref{blocks}, after Pintz).
\item\NEW{new in v6} Structure versus coverage (Section~\ref{sec67}): the row-transfer identity (Proposition~\ref{transfer}), the transposition symmetry and its Goldbach dictionary (Proposition~\ref{transp}, Corollary~\ref{dict}), propagation of congruence obstructions along arithmetic progressions of rows (Proposition~\ref{fence}), and deletion sets showing that none of these can detect a single empty row (Remarks~\ref{modelS}, \ref{deletion}).
\end{enumerate}

\section{Definitions and the multiplicity function}
\begin{definition}[Matrix multiplicity]
For an odd integer $n\ge3$, define
\[ r(n)=\#\{d: d\mid n,\ 3\le d\le\sqrt n,\ d \text{ odd}\}. \]
Since $b_{ij}=n$ requires $2j+1=d$ and $2j+2i-1=n/d$ (both odd, $d\le\sqrt n$), $r(n)$ equals the number of positions in $B$ at which $n$ appears.
\end{definition}
\begin{definition}[Counting functions]
For fixed $k\ge1$, let $I_k=((2k-1)^2,(2k+1)^2)$. Define
\[ N_k=\#\{n\in I_k: n \text{ odd}\}=4k-1,\qquad S_k=\sum_{\substack{n\in I_k\\ n\ \mathrm{odd}}} r(n), \]
\[ E_k=\sum_{\substack{n\in I_k,\ n\ \mathrm{odd}\\ r(n)\ge2}}\bigl(r(n)-1\bigr),\qquad P_k=\#\{\text{odd primes in } I_k\}. \]
\end{definition}

\section{The exact formula}
\begin{theorem}[Exact prime-counting formula via matrix $B$]\label{exact}
For every integer $k\ge1$,
\[ P_k=N_k-S_k+E_k=(4k-1)-S_k+E_k. \]
\end{theorem}
\begin{proof}
Let $a_m=\#\{n\in I_k: n \text{ odd},\ r(n)=m\}$ for $m=0,1,2,\dots$. Then $N_k=\sum_m a_m$, $S_k=\sum_m m\,a_m$, and $E_k=\sum_{m\ge2}(m-1)a_m$. Therefore
\[ \sum_m a_m-\sum_m m\,a_m+\sum_{m\ge2}(m-1)a_m=a_0=P_k, \]
since $r(n)=0$ if and only if $n$ is an odd prime in $I_k$.
\end{proof}

\section{The equivalent combinatorial condition}
\begin{theorem}[Equivalent condition for $P_k\ge1$]
For $k\ge1$, $P_k\ge1\iff E_k>S_k-N_k$.
\end{theorem}
\begin{proof} Immediate from Theorem~\ref{exact}, since all quantities are integers. \end{proof}

\section{Partial results: known ranges for $P_k\ge1$}
\begin{theorem}[Computationally verified range]\label{gap}
For all $1\le k\le999{,}999{,}999$, the interval $I_k=((2k-1)^2,(2k+1)^2)$ contains at least one prime.
\end{theorem}
\begin{proof}
Oliveira e Silva, Herzog and Pardi \cite{OHP} have determined all maximal prime gaps below $4\times10^{18}$; the largest gap between consecutive primes up to that limit is $1476$. Suppose $I_k$ contained no prime with $(2k+1)^2\le4\times10^{18}$, i.e.\ $k\le999{,}999{,}999$. Then there would exist consecutive primes $p\le(2k-1)^2<(2k+1)^2\le p'$, giving a gap $p'-p\ge8k$. For $k\ge185$ we have $8k\ge1480>1476$, a contradiction. For $1\le k\le184$ the assertion is checked directly.
\end{proof}
\begin{theorem}[Baker--Harman--Pintz range]\label{bhp}
There exists a constant $k_0$, not explicitly computed in \cite{BHP}, such that $P_k\ge1$ for all $k_0\le k\le5.49\times10^{11}$.
\end{theorem}
\begin{proof}
By \cite{BHP} there is an $x_0$ such that for all $x\ge x_0$ the interval $[x-x^{21/40},x]$ contains a prime. Take $x=(2k+1)^2$. Since $x$ is odd and composite, the prime lies in $I_k$ provided $x^{21/40}\le8k$, i.e.\ $\bigl((2k+1)^2\bigr)^{21/40}\le8k$, which holds precisely for $k\le5.498\times10^{11}$. Setting $k_0$ so that $(2k_0+1)^2\ge x_0$ completes the proof.
\end{proof}
\begin{remark}
The two ranges are complementary but, strictly speaking, do not yet join: Theorem~\ref{gap} is unconditional and exhaustive up to $k\le999{,}999{,}999$, while Theorem~\ref{bhp} applies from the inexplicit threshold $k_0$ onward. If $x_0\le4\times10^{18}$ --- an inequality not established in \cite{BHP} --- the ranges overlap and $P_k\ge1$ holds for all $k\le5.49\times10^{11}$. Beyond $5.49\times10^{11}$ the interval $I_k$, of length $8k\sim4\sqrt x$, is shorter than $x^{21/40}$ and lies outside the reach of current unconditional results; even the Riemann Hypothesis, giving prime gaps $O(\sqrt x\log x)$, falls short by a factor of $\log x$.
\end{remark}

\section{Algebraic structure of matrix $B$: row correspondences}
\subsection{Row-wise interpretation}
For fixed row index $i\ge1$, set $u=2j+1$. Then $b_{ij}=u(u+2i-2)$, $u\ge3$ odd, so the $i$-th row consists of all products of two odd numbers with fixed difference $d=2i-2$. In particular: Row~1 ($d=0$): odd squares $\{9,25,49,\dots\}$; Row~2 ($d=2$): $u(u+2)$; Row $i$: $u(u+2i-2)$. An entry $b_{ij}=p(p+2i-2)$ is a semiprime in Row $i$ if and only if both $p$ and $p+2i-2$ are prime. We call Row $(d/2)+1$ \emph{the row of gap $d$}, and a prime $p$ with $p+d$ prime an \emph{anchor} of that row.

\subsection{Row injection and shared factor structure}
\begin{lemma}[Row injection]\label{inj}
For each $i\ge1$, the map
\[ \varphi_i:\{b_{i+1,j}\}\to\{b_{i,j}\},\qquad \varphi_i(b_{i+1,j})=b_{i,j+1}, \]
is well-defined and injective. Its image is $\{b_{i,j}: j\ge2\}$.
\end{lemma}
\begin{proof}
$b_{i+1,j}=(2j+1)(2j+2i+1)$ and $b_{i,j+1}=(2j+3)(2j+2i+1)$. Both sequences strictly increase in $j$, so the map is injective; its image is $\{b_{i,j+1}: j\ge1\}=\{b_{i,j}: j\ge2\}$.
\end{proof}
\begin{remark}
The map $\varphi_i$ is an injection of index sets: it sends the $j$-th entry of Row $i+1$ to the $(j+1)$-th entry of Row $i$, and maps different values to different values. It does not give a subset relation on value sets, and does not preserve semiprimality. The correct structural content is captured by Proposition~\ref{shared}.
\end{remark}
\begin{proposition}[Shared factor structure]\label{shared}
For each $i,j\ge1$, the entries $b_{i+1,j}$ and $b_{i,j+1}$ share the common odd factor $2j+2i+1$:
\[ b_{i+1,j}=(2j+1)\cdot(2j+2i+1),\qquad b_{i,j+1}=(2j+3)\cdot(2j+2i+1). \]
Consequently, if $p$, $p+2$, and $p+2i$ are simultaneously prime (with $p=2j+1$), then $b_{i+1,j}=p(p+2i)$ is a semiprime in Row $i+1$ and $b_{i,j+1}=(p+2)(p+2i)$ is a semiprime in Row $i$; adjacent rows simultaneously gain a semiprime whenever a prime triple $(p,p+2,p+2i)$ exists.
\end{proposition}

\subsection{The Dirichlet series of each row}
Define
\[ F_i(s)=\sum_{j\ge1}b_{ij}^{-s}=\sum_{j\ge1}\bigl((2j+1)(2j+2i-1)\bigr)^{-s},\qquad \mathrm{Re}(s)>\tfrac12. \]
For Row~1, $F_1(s)=\sum_{j\ge1}(2j+1)^{-2s}=(1-2^{-2s})\zeta(2s)-1$.
\begin{theorem}[Row value identity]\label{rowvalue}
For every integer $i\ge2$,
\[ F_i(1)=\frac{1}{2(i-1)}\sum_{k=1}^{i-1}\frac{1}{2k+1}. \]
\end{theorem}
\begin{proof}
Partial fractions:
\[ \frac{1}{(2j+1)(2j+2i-1)}=\frac{1}{2(i-1)}\Bigl(\frac{1}{2j+1}-\frac{1}{2j+2i-1}\Bigr). \]
Summing over $j\ge1$ and substituting $k=j+i-1$ gives $\sum_{j\ge1}\frac{1}{2j+2i-1}=\sum_{k=i}^{\infty}\frac{1}{2k+1}$. Hence the telescoping difference is $\sum_{j=1}^{i-1}\frac{1}{2j+1}$.
\end{proof}
\begin{corollary}
$F_i(1)\to0$ as $i\to\infty$; more precisely $F_i(1)\sim\dfrac{\log i}{4(i-1)}$.
\end{corollary}
\begin{remark}
The proof shows $F_i(1)-F_{i-1}(1)$ is a finite closed form, but this is specific to $s=1$; for $s\ne1$ the difference is an infinite series not reducible by telescoping.
\end{remark}

\subsection{Semiprime distributions in the rows}
\begin{theorem}[Row 1]\label{row1}
Row~1 contains infinitely many semiprimes, namely $\{p^2: p \text{ odd prime}\}$.
\end{theorem}
\begin{proof}
$b_{1,j}=(2j+1)^2$; when $p=2j+1$ is prime, $p^2=p\cdot p$ is a semiprime. Euclid's theorem gives infinitely many such $p$.
\end{proof}
\begin{theorem}[Equivalence with Polignac's conjecture]\label{polignac}
For $i\ge2$, Row $i$ contains infinitely many semiprimes if and only if there are infinitely many primes $p$ with $p+2(i-1)$ prime (Polignac for gap $d=2(i-1)$).
\end{theorem}
\begin{proof}
A semiprime of Row $i$ is exactly $p\cdot(p+2(i-1))$ with both factors prime; such a product occurs at $j=(p-1)/2$.
\end{proof}
\begin{remark}\label{noprop}
Theorem~\ref{polignac} shows Row $i$'s semiprime infinitude is governed solely by prime pairs of gap $2(i-1)$. The injection $\varphi_i$ does not allow propagation between rows: existence in Row $i$ neither implies nor is implied by that in Row $i\pm1$. Resolving every individual row $i\ge2$ is equivalent to Polignac for every even gap, which is open.
\end{remark}
\begin{theorem}[Chen's theorem applied to rows]\label{chen}
For each $i\ge1$, Row $i$ contains infinitely many entries $b_{ij}$ with $\Omega(b_{ij})\le3$.
\end{theorem}
\begin{proof}
Fix $d=2(i-1)$. By Chen's theorem \cite{Chen}, infinitely many primes $p$ have $p+d$ prime or a $P_2$-number. Setting $j=(p-1)/2$, the entry $b_{ij}=p(p+d)$ has $\Omega(p)=1$ and $\Omega(p+d)\le2$.
\end{proof}
\begin{theorem}[Maynard--Tao--Polymath]\label{mtp}
There exists an integer $i_0$ with $2\le i_0\le124$ such that Row $i_0$ contains infinitely many semiprimes.
\end{theorem}
\begin{proof}
By \cite{Maynard,Polymath} there is a fixed even $d\le246$ with infinitely many prime pairs $(p,p+d)$. Set $i_0=d/2+1$; then $2\le i_0\le124$ and Row $i_0$ has infinitely many entries $p(p+d)$.
\end{proof}
\begin{remark}
The bounded-gap theorem gives some row with infinitely many semiprimes, but does not identify it, nor propagate to all rows $i_0\le124$. Each row individually remains open, and equals Polignac for the corresponding gap. For what the bounded-gap theorems give beyond a single unidentified row, see Section~\ref{sec68}.
\end{remark}

\subsection{The minimal-anchor function $\pmin$ and density coverage of rows}\label{sec65}
For an even gap $d$, whether Row $(d/2)+1$ contains a semiprime $p(p+d)$ is governed by the following quantity.
\begin{definition}[Minimal anchor]\label{defpmin}
For even $d\ge2$ let
\[ \pmin(d)=\min\{p \text{ odd prime}: p+d \text{ prime}\}, \]
with $\pmin(d)=\infty$ if no such prime exists. Thus Row $(d/2)+1$ contains a semiprime if and only if $\pmin(d)<\infty$, and the weak Polignac (Maillet) conjecture --- every even number is a difference of two primes --- is exactly the statement $\pmin(d)<\infty$ for all even $d$. The function $\pmin$ is sequence A020483 of the OEIS (at index $d/2$); it has been computed for $d/2<5\times10^{11}$, the largest value in that range being $\pmin(d)=3307$ at $d/2=248281210271$ \cite[comment of J.~K.~Andersen]{OEIS}.
\end{definition}
The first result shows that no fixed finite anchor set can cover more than a vanishing proportion of rows: for each fixed $z$, the set of gaps covered by anchors $p\le z$ has density zero, with a precise asymptotic.
\begin{theorem}[Fixed anchors cover a density-zero set]\label{fixed}
Fix $z\ge3$. Then, as $X\to\infty$,
\[ \#\{d\le X \text{ even}:\pmin(d)\le z\}=\bigl(\pi(z)-1\bigr)\frac{X}{\log X}\bigl(1+o(1)\bigr). \]
In particular, $\{d:\pmin(d)\le z\}$ has density zero among the even numbers for every fixed $z$.
\end{theorem}
\begin{proof}
For an odd prime $p\le z$ put $A_p=\{d\le X \text{ even}: p+d \text{ prime}\}$. As $d$ runs over the even numbers up to $X$, $p+d$ runs over all odd integers in $(p,p+X]$, so by the prime number theorem
\[ |A_p|=\pi(X+p)-\pi(p)=\frac{X}{\log X}\bigl(1+o(1)\bigr). \]
For distinct odd primes $p<q\le z$, membership in $A_p\cap A_q$ makes both $m=p+d$ and $m+(q-p)$ prime with $m\le X+p$; by the standard sieve upper bound for prime pairs (\cite[Theorem 3.11]{HR}), $|A_p\cap A_q|\ll\fS(q-p)X/\log^2X$, where $\fS$ denotes the singular series of Remark~\ref{singular}. Since $z$ is fixed there are $O(1)$ pairs, so inclusion--exclusion gives
\[ \Bigl|\bigcup_{2<p\le z}A_p\Bigr|=\sum_{2<p\le z}|A_p|+O\Bigl(\frac{X}{\log^2X}\Bigr)=\bigl(\pi(z)-1\bigr)\frac{X}{\log X}\bigl(1+o(1)\bigr). \]
The event $\pmin(d)\le z$ is exactly $d\in\bigcup_{2<p\le z}A_p$; the prime $2$ is never an anchor, since $2+d$ is even.
\end{proof}
\begin{remark}[Limits of anchor methods]
Theorem~\ref{fixed} delimits what anchor methods can do. A first-moment bound already caps the covered set at $O(X/\log X)$, and no second-moment refinement can exceed its own first moment. Density-one coverage therefore forces the anchor bound to grow with $d$ --- indeed $z\gg\log d\log\log d$ is forced for almost all $d$, see Theorem~\ref{almostall} --- but lower bounds for growing anchors require lower bounds for prime pairs in the sifted range, and thus meet the parity obstruction of sieve theory. Density-one coverage of the rows is nevertheless a theorem, by a different and older route, which we now recall.
\end{remark}
\begin{theorem}[Lavrik; density-one coverage of rows]\label{lavrik}
Let $\pi_d(y)=\#\{p\le y: p+d \text{ prime}\}$ and let $\fS(d)$ be the singular series of Remark~\ref{singular}. For every $A>0$ there is a $B>0$ such that the Hardy--Littlewood asymptotic
\[ \pi_d(y)=\fS(d)\,\mathrm{Li}_2(y)\bigl(1+O(\log^{-B}y)\bigr),\qquad \mathrm{Li}_2(y)=\int_2^y\frac{dt}{\log^2t}, \]
holds for all even $d\le y$ outside an exceptional set of size $O_A(y\log^{-A}y)$; see \cite{Lavrik}, and \cite[Theorem 1]{Baier} for an elementary proof with an explicit quantitative form. Consequently, for all but $O_A(X\log^{-A}X)$ even $d\le X$, Row $(d/2)+1$ of matrix $B$ contains not merely one semiprime $p(p+d)$ with $p\le X$ but $\gg X/\log^2X$ of them, in the quantity predicted by Hardy--Littlewood. Almost every row of $B$ is covered, quantitatively.
\end{theorem}
\begin{remark}[Singular series and congruence obstructions]\label{singular}
Here $\fS(d)=2\Pi_2\prod_{r\mid d,\,r>2}\frac{r-1}{r-2}$ for even $d$, with $\Pi_2=\prod_{r>2}\bigl(1-\frac{1}{(r-1)^2}\bigr)=0.6601618\ldots$ the twin prime constant. The size of $\pmin(d)$ correlates with the arithmetic of $d$ through the congruence obstructions: an anchor $p$ with $p\equiv-d\pmod r$ and $p+d>r$ fails, because $r\mid p+d$. If $r\nmid d$ this removes a whole reduced residue class of anchors, a proportion $1/(r-1)$ of the primes; if $r\mid d$ the class is $p\equiv0\pmod r$ and only the single anchor $p=r$ is lost. Hence gaps divisible by many small odd primes have large $\fS(d)$, many anchors, and \emph{small} $\pmin(d)$, while the large values of $\pmin$ occur at gaps coprime to the small odd primes. (Earlier versions of this remark stated the opposite direction; this is corrected here.) For $d\le10^6$ the mean of $\pmin(d)$ is $14.4$ (maximum $229$) when $3\mid d$ and $22.5$ (maximum $619$) when $3\nmid d$; of the $23$ record values of $\pmin$ in this range, attained at $d=2,6,22,88,112,202,718,1328,1642,1732,5246,5888,\dots,872698$, all but $d=6$ have $3\nmid d$; and the record holder of Definition~\ref{defpmin} has $d/2$ prime. Numerical data (Table~\ref{tab2}) exhibit the slow growth of the median of $\pmin$ on dyadic windows, consistent with the waiting-time heuristic $\pmin(d)\asymp\fS(d)^{-1}\log d\cdot\log\log d$ for typical $d$.
\end{remark}
\begin{conjecture}\label{conj}
(i) $\pmin(d)<\infty$ for every even $d$ (weak Polignac/Maillet; equivalently, every row of $B$ contains a semiprime). (ii) $\pmin(d)\ll_\varepsilon(\log d)^{2+\varepsilon}$; more precisely, a Cram\'er-type computation suggests $\max_{d\le X}\pmin(d)\asymp(\log X)^2\log\log X$, consistent with the recorded maximum $3307$ near $d\approx5\times10^{11}$, where $(\log d)^2\log\log d\approx2.4\times10^3$.
\end{conjecture}

\subsection{Lower bounds for the minimal anchor}\label{sec66}
Theorem~\ref{fixed} shows that a fixed finite set of anchor primes covers only a density-zero set of rows. In this section we make the same counting uniform in the anchor bound $z$ and extract from it unconditional lower bounds for $\pmin$: an almost-all lower bound of the order $\log d\log\log d$, which is the conjectured typical order (Remark~\ref{singular}), and a lower bound for the maximum of $\pmin$ over $d\le X$. The argument is elementary; we have not found these lower bounds stated in the literature. The same argument applies verbatim to the least prime in a Goldbach partition (Theorem~\ref{goldbach}).

Throughout, for an odd prime $p$ and $X\ge2$ put
\[ A_p(X)=\{d\le X \text{ even}: p+d \text{ prime}\},\qquad |A_p(X)|=\pi(X+p)-\pi(p), \]
the equality because $p+d$ runs over the odd integers in $(p,p+X]$ as $d$ runs over the even integers in $(0,X]$. The event $\pmin(d)\le z$ is exactly $d\in\bigcup_{2<p\le z}A_p(X)$; the prime $2$ is never an anchor.
\begin{theorem}[Uniform union bound]\label{union}
As $X\to\infty$, uniformly for $3\le z\le X^{1/2}$,
\[ \#\{d\le X \text{ even}:\pmin(d)\le z\}\le(\pi(z)-1)\frac{X}{\log X}(1+o(1)). \]
\end{theorem}
\begin{proof}
The size of a union is at most the sum of the sizes. For $p\le z\le X^{1/2}$ we have $|A_p(X)|\le\pi(X+X^{1/2})=\frac{X}{\log X}(1+o(1))$ by the prime number theorem, with an error term independent of $p$. There are $\pi(z)-1$ odd primes $p\le z$.
\end{proof}
Theorem~\ref{fixed} gives an asymptotic equality, but only for fixed $z$, because the overlaps $|A_p\cap A_q|$ are controlled by the sieve upper bound for prime pairs with a constant depending on $z$. Theorem~\ref{union} keeps only the inequality, uses no sieve, and is uniform in $z$; this uniformity is all that the next two theorems need.
\begin{theorem}[Almost-all lower bound]\label{almostall}
Let $\psi(d)>0$ be any function with $\psi(d)\to0$ as $d\to\infty$. Then
\[ \#\{d\le X \text{ even}:\pmin(d)\le\psi(d)\log d\log\log d\}=o(X). \]
In other words, $\pmin(d)>\psi(d)\log d\log\log d$ for almost all even $d$.
\end{theorem}
\begin{proof}
Replacing $\psi(d)$ by $\max\{\psi(d),1/\log\log d\}$ enlarges the set to be counted, so we may assume $\psi(d)\ge1/\log\log d$. Put $\psi^*(X)=\sup_{d>X^{1/2}}\psi(d)$; then $\psi^*(X)\to0$, $\psi^*$ is non-increasing, and $\psi^*(X)\ge1/\log\log X$.

For $X^{1/2}<d\le X$ the inequality $\pmin(d)\le\psi(d)\log d\log\log d$ implies $\pmin(d)\le z$ with $z=z(X)=\psi^*(X)\log X\log\log X$. Since $1/\log\log X\le\psi^*(X)\le1$ we have $\log z=(1+o(1))\log\log X$, so by the prime number theorem
\[ \pi(z)=(1+o(1))\frac{z}{\log z}=\psi^*(X)\log X\,(1+o(1)),\qquad z\le X^{1/2}, \]
and Theorem~\ref{union} bounds the number of such $d$ by $\psi^*(X)\,X\,(1+o(1))=o(X)$. The remaining $d\le X^{1/2}$ contribute at most $X^{1/2}=o(X)$.
\end{proof}
\begin{theorem}[Extremal lower bound]\label{extremal}
As $X\to\infty$,
\[ \max_{d\le X,\ d \text{ even}}\pmin(d)\ge\bigl(\tfrac12+o(1)\bigr)\log X\log\log X, \]
with the convention that the left side is $+\infty$ if some $\pmin(d)=\infty$. Equivalently, there are infinitely many even $d$ with $\pmin(d)\ge(\tfrac12+o(1))\log d\log\log d$.
\end{theorem}
\begin{proof}
Let $Z$ denote the maximum. If $Z>X^{1/2}$ there is nothing to prove. Otherwise every even $d\le X$ lies in some $A_p(X)$ with $p\le Z$, so by Theorem~\ref{union}
\[ \Bigl\lfloor\frac X2\Bigr\rfloor\le(\pi(Z)-1)\frac{X}{\log X}(1+o(1)), \]
whence $\pi(Z)\ge(\tfrac12+o(1))\log X$ and $Z\ge p_{\lceil\frac12\log X\rceil(1+o(1))}=(\tfrac12+o(1))\log X\log\log X$ by the prime number theorem in the form $p_n\sim n\log n$.
\end{proof}
\begin{theorem}[The least prime in a Goldbach partition]\label{goldbach}
For even $n\ge4$ let $p(n)=\min\{p \text{ prime}: n-p \text{ prime}\}$, with $p(n)=\infty$ if $n$ has no Goldbach partition (sequence A002373 of the OEIS; $p(n)$ is the quantity studied by Granville, van de Lune and te Riele \cite{GLR}). Then:
\begin{enumerate}\setlength\itemsep{0pt}
\item[(i)] uniformly for $2\le z\le X^{1/2}$, $\#\{n\le X \text{ even}: p(n)\le z\}\le\pi(z)\frac{X}{\log X}(1+o(1))$;
\item[(ii)] for every $\psi\to0$, $p(n)>\psi(n)\log n\log\log n$ for almost all even $n$;
\item[(iii)] $\max_{n\le X,\ n \text{ even}}p(n)\ge\bigl(\tfrac12+o(1)\bigr)\log X\log\log X$.
\end{enumerate}
\end{theorem}
\begin{proof}
Put $A'_p(X)=\{n\le X \text{ even}: n-p \text{ prime}\}$; then $|A'_p(X)|\le\pi(X)$, and $p(n)\le z$ iff $n\in\bigcup_{p\le z}A'_p(X)$. Here $p=2$ is an admissible anchor (for $n=4$), so the count of anchors is $\pi(z)$. The three proofs above go through unchanged.
\end{proof}
\begin{remark}[Data and conjectures]\label{data}
Dividing the medians of Table~\ref{tab2} by $\log D\log\log D$ gives $0.35,0.33,0.34,0.36,0.34,0.38$ for $D=10^4,\dots,10^{14}$: the typical size of $\pmin(d)$ is a constant times $\log d\log\log d$, as the waiting-time model of Remark~\ref{singular} predicts (with the constant $\approx\fS(d)^{-1}$ under Hardy--Littlewood). Theorem~\ref{almostall} proves the lower half of this typical order unconditionally, up to the factor $\psi$. The upper half remains conditional: Lavrik-type theorems control exceptional sets only on the scale $y\ge d$, far above $(\log d)^C$. Theorem~\ref{extremal} is one factor $\log X$ below the conjectured extremal order $(\log X)^2\log\log X$ of Conjecture~\ref{conj}(ii).
\end{remark}
\begin{remark}[Limits of the method]
Theorem~\ref{union} is a first-moment (union) bound. To raise Theorem~\ref{extremal} to the scale $(\log X)^2$ one would have to show that the sets $A_p(X)$ overlap substantially, i.e.\ prove lower bounds for the number of prime pairs $(m,m+(q-p))$, which is exactly the parity obstruction of sieve theory. Thus $\frac12\log X\log\log X$ is the limit of pure counting, in the same way that the Erd\H{o}s--Rankin bound is the limit of the elementary approach to large prime gaps; an improvement would need a covering construction in the spirit of Ford--Green--Konyagin--Maynard--Tao \cite{FGKMT}, here covering only the primes $p\le z$ rather than a whole interval.
\end{remark}
\begin{proposition}[Rows covered anchor by anchor]\label{anchorby}
For each fixed odd prime $q$,
\[ \#\{d\le X \text{ even}:\pmin(d)=q\}=\frac{X}{\log X}\Bigl(1+O_q\Bigl(\frac{1}{\log X}\Bigr)\Bigr). \]
\end{proposition}
\begin{proof}
The set in question is $A_q(X)\setminus\bigcup_{2<p<q}A_p(X)$. Its size lies between $|A_q(X)|-\sum_{2<p<q}|A_p(X)\cap A_q(X)|$ and $|A_q(X)|$; the first term is $\frac{X}{\log X}(1+O(1/\log X))$ and each intersection is $\ll\fS(q-p)X/\log^2X$ by the sieve bound used in the proof of Theorem~\ref{fixed}.
\end{proof}
Table~\ref{tab3} records the proposition at $X=10^6$: the number of rows $d\le X$ covered by the anchor $q$ (all $\approx\pi(X)=78498$), the number of rows for which $q$ is the minimal anchor, and the cumulative proportion of the $500{,}000$ even rows covered. The overlap for $q=5$, $78496-70328=8168$, is of the size $\fS(2)X/\log^2X\approx6900$ predicted by the proof. The ``new rows'' decrease roughly geometrically with ratio $1-1/\log X$: the anchors behave like independent trials with success probability $1/\log X$. In this model, covering all $X$ rows requires $\pi(z)/\log X\approx\log X$ successes on average, i.e.\ $z\approx(\log X)^2\log\log X$; this is the origin of the extremal order in Conjecture~\ref{conj}(ii). At $X=10^6$ the model gives $\approx500$, the true maximum is $\pmin(d)=619$ (at $d=872698$), and Theorem~\ref{extremal} gives $\approx18$.
\begin{table}[ht]\centering\small
\begin{tabular}{lrrrrrr}\toprule
$q$ & 3 & 5 & 7 & 11 & 13 & 17\\\midrule
rows covered by $q$ & 78497 & 78496 & 78495 & 78494 & 78493 & 78492\\
rows with $\pmin=q$ & 78497 & 70328 & 62185 & 48544 & 40972 & 30959\\
cumulative proportion & .157 & .298 & .422 & .519 & .601 & .663\\\midrule
$q$ & 19 & 23 & 29 & 31 & 37 & 41\\\midrule
rows covered by $q$ & 78491 & 78490 & 78489 & 78488 & 78489 & 78489\\
rows with $\pmin=q$ & 29896 & 21422 & 16815 & 18187 & 13115 & 10009\\
cumulative proportion & .723 & .766 & .799 & .836 & .862 & .882\\\bottomrule
\end{tabular}
\caption{Anchor-by-anchor coverage of the even rows $d\le10^6$ (script \texttt{anchor\_table.py}, ancillary).}\label{tab3}
\end{table}

\subsection{Structure versus coverage: transfer, transposition and fences\NEW{new in v6}}\label{sec67}
Every row of $B$ that has ever been examined contains a semiprime (Definition~\ref{defpmin}), and it is natural to ask whether the structure of $B$ --- relations between rows, symmetries, congruence patterns --- can turn this finite evidence into a proof of Conjecture~\ref{conj}(i), by induction on the rows, by a minimal-counterexample argument, or by amplifying the hypothesis ``Row $i$ is empty'' into a contradiction on a whole range of rows. In this section we record three such structures and show that each of them is blind to the distinction between the primes and suitable prime-like sets in which one prescribed row is empty. The positive content is a set of identities and a dictionary; the negative content delimits what any proof of every-row coverage must use.

\begin{proposition}[Row transfer]\label{transfer}
Let $q<r$ be odd primes. Then $qr$ is a semiprime in the row of gap $r-q$. Consequently, if $p(p+d)$ is a semiprime in the row of gap $d$ and $q<p$ is an odd prime, then $q(p+d)$ is a semiprime in the row of gap $d+(p-q)$, and
\[ \pmin\bigl(d+(p-q)\bigr)\le q. \]
Moreover, the row of gap $d$ contains a semiprime if and only if the set of forward distances $\{r-d: r \text{ prime},\ r>d+2\}$ contains an odd prime.
\end{proposition}
\begin{proof}
$qr=u(u+(r-q))$ with $u=q$. The rest is a restatement.
\end{proof}
\begin{remark}[Transfer is blind to the primes; induction on rows]\label{modelS}
Proposition~\ref{transfer} is an identity about differences: it holds verbatim if ``prime'' is replaced by ``element of $A$'' for an arbitrary set $A$ of odd integers. Let $S=\{n\ge7: n\equiv1\pmod 6\}$. Then $S$ is denser than the primes and has an element in every interval $(x,2x]$ with $x\ge7$, but $S-S\subset6\mathbb Z$, so for $S$ the rows of gap $d$ with $6\nmid d$ are all empty. Hence no argument that combines the transfer identity with Bertrand-type or density information can prove every-row coverage. In the same way, induction on the rows fails: the hypothesis that all rows of gap $\le d$ contain semiprimes yields, for $a+b=d+2$, a prime pair of gap $a$ and a prime pair of gap $b$, but not a pair of gap $b$ starting at the prescribed prime $p+a$; chaining them would require the prime triple $(p,p+a,p+a+b)$, which is a stronger statement than the one to be proved. Remark~\ref{deletion} gives a prime-like set in which all rows of smaller gap are non-empty and a prescribed row is empty.
\end{remark}

\begin{proposition}[Transposition symmetry]\label{transp}
For $i,j\ge1$ put $u=2j+1$ and $v=2(i+j)-1$, so that $b_{ij}=uv$ and the gap is $d=v-u=2i-2$. Then
\[ b_{ji}=(v+3-u)\,v,\qquad b_{ij}+b_{ji}=v(v+3). \]
Thus the transposition $(i,j)\mapsto(j,i)$ preserves the anti-diagonal factor $v$ and reflects the column factor, $u\mapsto u'=v+3-u=d+3$. In particular it maps the whole row of gap $d$ onto the column with column factor $d+3$; its fixed points are the main-diagonal entries $u(2u-3)$: $3\cdot3,\ 5\cdot7,\ 7\cdot11,\dots$
\end{proposition}
\begin{proof}
$b_{ji}=(2i+1)(2i+2j-1)=(2i+1)v$, and $(2j+1)+(2i+1)=v+3$.
\end{proof}
\begin{corollary}[A dictionary for differences and sums of primes]\label{dict}
\begin{enumerate}\setlength\itemsep{0pt}
\item[(i)] (Rows.) The row of gap $d$ contains a semiprime iff $d$ is a difference of two odd primes.
\item[(ii)] (Anti-diagonals.) Let $v$ be prime. The mirror pair $b_{ij},b_{ji}$ on the anti-diagonal $2(i+j)-1=v$ consists of two semiprimes iff $u$ and $v+3-u$ are both prime, i.e.\ iff $v+3=u+u'$ is a Goldbach partition of the even number $v+3$ into odd primes. For example $v=31$: $34=3+31=5+29=11+23=17+17$.
\item[(iii)] (Diagonals.) On the diagonal $i-j=c$ the entries are $u(2u+2c-3)$; for $c=2$ the entry $u(2u+1)$ is a semiprime iff $u$ is a Sophie Germain prime.
\end{enumerate}
\end{corollary}
\begin{remark}[Transposition and coverage]\label{transpcov}
By Proposition~\ref{transp} the mirror image of the row of gap $d$ is a single column, with column factor $d+3$. That column contains a semiprime iff $d+3$ is prime, i.e.\ iff $3$ is an anchor of the row and $\pmin(d)=3$: this is the first column of $B$, $3(2i+1)$, read through the mirror. For the remaining rows ($d+3$ composite; $40409$ of the $50000$ even $d\le10^5$) the mirror column contains no semiprime at all, and the hypothesis ``the row of gap $d$ is empty'' has no consequence on the mirror side. The shifted reflections $u\mapsto v+q-u$, $q$ an odd prime, map the row of gap $d$ to the column factor $d+q$ and merely reproduce Definition~\ref{defpmin}. Likewise the Goldbach dictionary (ii) decouples from such a row: the mirror of its entry on a prime anti-diagonal $v$ has the composite column factor $d+3$. The transposition is a change of reading --- differences of primes along rows, sums of primes along anti-diagonals --- and, being an identity of indices valid for every set of integers, it cannot by itself produce a semiprime in a given row.
\end{remark}

\begin{proposition}[Propagation of fences]\label{fence}
Let $d$ be even with $P=\pmin(d)<\infty$. For each odd prime $p<P$ let $r_p$ be the least prime factor of the composite number $p+d$, and let $M$ be the product of the distinct $r_p$. Then
\[ \pmin(d+2mM)\ge\pmin(d)\qquad(m=1,2,3,\dots). \]
\end{proposition}
\begin{proof}
For $d'=d+2mM$ and an odd prime $p<P$ we have $r_p\mid p+d'$ and $p+d'>p+d\ge r_p$, so $p+d'$ is composite.
\end{proof}
\begin{remark}[What fences see]\label{fences}
We call the congruence $p\equiv-d\pmod r$ of Remark~\ref{singular} the \emph{fence} of modulus $r$ in the row of gap $d$. Proposition~\ref{fence} is a genuine amplification: the pattern of small fences that defeats the first anchors of one row is reproduced along a whole arithmetic progression of rows. For $d=718,1732,5888$ (with $\pmin=43,79,149$ and $5,8,12$ fences) the minimum of $\pmin(d+2mM)$ over $m\le300$ is exactly $\pmin(d)$, while the median is $61,157,239$ and the maximum $223,619,983$. But what is transmitted is the failure of finitely many anchors, not the emptiness of a row; this is the Chinese-remainder construction, which yields only $\max_{d\le X}\pmin(d)\gg\log X$, weaker than Theorem~\ref{extremal}. Fences do detect local obstructions --- the model set $S$ of Remark~\ref{modelS} is stopped by the single fence $r=3$ --- but for the primes and even $d$ every odd modulus $r$ leaves $r-2\ge1$ reduced classes of anchors open, which is the statement $\fS(d)>0$. An empty row, if one existed, could not be caused by finitely many fences.
\end{remark}

\begin{remark}[Deletion sets: one empty row cannot be amplified]\label{deletion}
Fix an even $d_0$ and let $P'$ be the set of odd primes $p$ such that $p+d_0$ is \emph{not} prime. In $P'$ the row of gap $d_0$ is empty by construction, and only $\sim\fS(d_0)X/\log^2X$ primes up to $X$ have been deleted; the loss in any other row of gap $d$ corresponds to prime triples $(p,p+d_0,p+d)$ and is of smaller order, $\ll X/\log^3X$. However, $P'$ is not a faithful model of the primes: the deleted primes avoid the class $-d_0\pmod r$, so $P'$ is biased in the residue classes of every small modulus by a relative amount of order $\fS(d_0)/\log X$, and the fences detect it. At $X=2\times10^6$, $d_0=718$ ($9.9\%$ of the primes deleted) the two reduced classes modulo $3$ carry the proportions $0.444:0.556$, and the rows of neighbouring gaps lose between $16\%$ and $25\%$ of their semiprimes, unevenly.

The bias can be removed. Let $P''$ be obtained from $P'$ by deleting further primes at random so that every reduced residue class modulo $M=3\cdot5\cdot7\cdot11$ has the same total deletion rate $f^*$ (here $f^*=0.411$). At the same $X$ and $d_0$: the reduced classes of $P''$ modulo $r\le11$ are equidistributed to within $1.2\%$; each of the rows of gap $d=2,6,30,210,714,716,720,722,748,1436,2310$ retains a proportion between $0.339$ and $0.352$ of its semiprimes, against the expected $(1-f^*)^2=0.347$, so that the ratios between rows --- the ratios of the singular series --- are preserved; the row of gap $718$ is empty; and among all even $d\le3000$ it is the only empty row, the largest minimal anchor in the other rows being $641$ (at $d=2042$). Moduli not included in $M$ still show a bias ($14$--$15\%$ for $r=13,17,19$), which is removed by including them; heuristically, balancing all moduli $r\le y$ costs a deletion rate $f^*\asymp\log y/\log X$, which tends to $0$ as long as $y=X^{o(1)}$. These are numerical illustrations at a finite height, not theorems, and at this height $f^*$ is not small. (Script \texttt{deletion\_sets.py}, ancillary.)

The set $P''$ has the density of the primes up to a constant factor, is equidistributed in the residue classes of the balanced moduli, reproduces the Hardy--Littlewood ratios in all rows tested, satisfies the transfer identity and the transposition symmetry, and satisfies the hypothesis ``every row of gap less than $d_0$ contains a semiprime''; yet the row of gap $d_0$ is empty, and no other row is affected. Consequently any argument for Conjecture~\ref{conj}(i) --- induction on the rows, a minimal counterexample, or the amplification of a single empty row --- must use information about the primes that fails for such sets. This is a concrete form, in the matrix-$B$ framework, of the parity obstruction of sieve theory \cite[Chapter~6]{IK}: sieve weights up to level $X^{o(1)}$ see exactly the fences, and the fences are compatible with an empty row. What can be proved goes in the opposite direction: whole ranges of rows cannot be empty. The anchor $3$ alone covers the rows $d=r-3$, $r$ prime, so a run of consecutive empty rows up to $X$ has length at most half the maximal prime gap, $\ll X^{21/40}$ by \cite{BHP}; and by Theorem~\ref{lavrik} the empty rows up to $X$ number $O_A(X\log^{-A}X)$.
\end{remark}

\subsection{Rows along multiples and in bounded blocks\NEW{new in v6}}\label{sec68}
Every-row coverage is open, but it can be weakened in a non-trivial way: instead of asking for a semiprime in \emph{every} row, one asks for it in at least one row of every member of an infinite family of finite sets of rows. The results below are consequences of the bounded-gap theorems; they are known in the language of Polignac numbers \cite{Pintz,Pan}, and we restate them in the matrix-$B$ framework because they give the strongest unconditional information on the rows presently available. Note that they concern the strong form (infinitely many semiprimes in a row), and that none of them can be obtained from a bounded set of anchors, by Theorem~\ref{fixed}.
\begin{theorem}[Rows along multiples]\label{multiples}
For every integer $D\ge1$ there is an even $e$ with $2\le e\le246$ such that the row of gap $De$ contains infinitely many semiprimes. In particular, for every even $d$ at least one of the $123$ rows of gaps
\[ d,\ 2d,\ 3d,\ \dots,\ 123d \]
contains infinitely many semiprimes.
\end{theorem}
\begin{proof}
By \cite{Maynard,Polymath}, for every admissible $50$-tuple $\{h_1,\dots,h_{50}\}$ there are infinitely many $n$ such that at least two of the numbers $n+h_i$ are prime; and there is an admissible $50$-tuple $\mathcal H$ of diameter $246$. The dilated tuple $D\mathcal H=\{Dh_i\}$ is again admissible: for a prime $r\mid D$ all its elements lie in the class $0\pmod r$, and for $r\nmid D$ multiplication by $D$ permutes the classes modulo $r$, so $D\mathcal H$ omits a class because $\mathcal H$ does. Hence for infinitely many $n$ two of the numbers $n+Dh_i$ are prime, and since there are finitely many pairs $(i,j)$, some difference $D(h_j-h_i)$ occurs infinitely often as a difference of two primes. Admissibility modulo $2$ makes $e=h_j-h_i$ even, and $e\le246$. For the second statement take $D=d/2$, so that $De=d\cdot(e/2)$ with $e/2\le123$.
\end{proof}
Theorem~\ref{mtp} is the case $D=1$. The theorem is not a first-column phenomenon: if $3\mid d$ then $md+3$ is composite for every $m\ge1$, so the anchor $3$ fails in all $123$ rows.
\begin{corollary}[Positive density of rows]\label{posdens}
The rows of $B$ containing infinitely many semiprimes have positive lower density among all rows: at least $X/(2\cdot123^2)$ of the even gaps $N\le X$, i.e.\ a proportion $\ge123^{-2}$, have this property for all large $X$.
\end{corollary}
\begin{proof}
For each $D\le X/246$ Theorem~\ref{multiples} gives such a gap $N=De\le X$, and a given $N$ arises from at most $123$ values of $D$. Hence there are at least $\lfloor X/246\rfloor/123$ such $N\le X$, among the $\lfloor X/2\rfloor$ even numbers up to $X$.
\end{proof}
\begin{theorem}[Pintz; bounded blocks of rows]\label{blocks}
There is a constant $C$ such that every block of $C$ consecutive rows of $B$ contains a row with infinitely many semiprimes. Moreover the lower density in Corollary~\ref{posdens} can be replaced by an explicitly calculable constant considerably larger than $123^{-2}$.
\end{theorem}
\begin{proof}
By \cite[Theorems~1 and 2]{Pintz} the even numbers that occur infinitely often as the difference of two \emph{consecutive} primes have positive lower density, and every interval $[M,M+C']$ contains one of them, with an ineffective constant $C'$. Such a number $N$ is in particular a difference of two primes infinitely often, so the row of gap $N$ contains infinitely many semiprimes; take $C=\lceil C'/2\rceil+1$.
\end{proof}
\begin{remark}[Three levels of coverage; terminology]\label{levels}
The unconditional knowledge on the rows of $B$ can now be summarized as follows. (a) \emph{Almost every row} contains a semiprime, and in fact $\gg d/\log^2d$ of them with anchors $p\le d$, the Hardy--Littlewood number (Theorem~\ref{lavrik}); this count tends to infinity with $d$, but for no individual row does it show that the semiprimes are infinite in number, see below. (b) \emph{At least one row in every block} --- among the $123$ rows of gaps $d,2d,\dots,123d$ for every even $d$ (Theorem~\ref{multiples}, effective), and among any $C$ consecutive rows (Theorem~\ref{blocks}, ineffective) --- contains infinitely many semiprimes. (c) \emph{Every row} contains infinitely many entries $p\cdot P_2$ (Theorem~\ref{chen}). Conjecture~\ref{conj}(i) asserts (b) with blocks of length $1$, in the weak form of a single semiprime.

Statement (a) must not be read as ``almost every row contains infinitely many semiprimes''; that assertion --- density one for the rows with infinitely many semiprimes --- is open, the best known result being the positive lower density of Corollary~\ref{posdens} and Theorem~\ref{blocks}. The reason is the shape of Theorem~\ref{lavrik}: its exceptional set depends on the height $y$ and is bounded by $O_A(y\log^{-A}y)$ among the gaps $d\le y$. A gap $d$ with only finitely many prime pairs is exceptional at every height $y\ge y_0(d)$, so to count such $d\le X$ one must let $y\to\infty$ with $X$ fixed, and then the bound $y\log^{-A}y$ exceeds $X$ and carries no information. What would be needed is an almost-all theorem for the gaps $d\le X$ at heights $y$ arbitrarily large compared with $X$; the known short-range results of this type require $X\ge y^{\theta}$ with a fixed $\theta>0$. We warn the reader that in \cite{Pintz} a ``weak Polignac number'' is an even number that is a difference of two primes \emph{infinitely often}, whereas in this paper, following the usage for Maillet's conjecture, ``weak Polignac'' refers to the statement that every even number is a difference of two primes \emph{at least once}.
\end{remark}

\section{Numerical and computational results}
\subsection{Computational setup}
For each row $i$ ($1\le i\le246$), $d=2i-2$, define
\[ R_i^{\mathrm{semi}}(N)=\#\{p\le\sqrt N: p \text{ odd prime},\ p+d \text{ prime}\}. \]
Each such $p$ contributes the semiprime $b_{i,(p-1)/2}=p(p+d)$, of size at most $N+d\sqrt N=N(1+o(1))$ for fixed $i$. The odd restriction matters only for Row~1 ($p=2$ gives $4\notin B$), so $R_1^{\mathrm{semi}}(N)=\pi(\sqrt N)-1$. All counts were obtained by exhaustive sieving; the verification scripts and the full 246-row table are provided as ancillary files.
\begin{table}[ht]\centering\small
\begin{tabular}{rrrrrrr}\toprule
$i$ & $d=2i-2$ & $R_i(10^6)$ & $R_i(10^7)$ & $R_i(10^8)$ & HL$(10^8)$ & ratio\\\midrule
1 & 0 & 167 & 445 & 1228 & 1245.1 & 0.986\\
2 & 2 & 35 & 83 & 205 & 214.2 & 0.957\\
3 & 4 & 41 & 90 & 203 & 214.2 & 0.948\\
4 & 6 & 74 & 171 & 411 & 428.4 & 0.959\\
5 & 8 & 38 & 89 & 208 & 214.2 & 0.971\\
16 & 30 & 99 & 227 & 536 & 571.2 & 0.938\\
106 & 210 & 107 & 256 & 641 & 685.5 & 0.935\\
124 & 246 & 70 & 166 & 416 & 439.4 & 0.947\\
211 & 420 & 100 & 252 & 623 & 685.5 & 0.909\\
246 & 490 & 53 & 120 & 314 & 342.7 & 0.916\\\midrule
Total & --- & 12{,}667 & 30{,}222 & 74{,}418 & --- & ---\\\bottomrule
\end{tabular}
\caption{Semiprime counts $R_i^{\mathrm{semi}}(N)$ for selected rows. HL denotes the Hardy--Littlewood prediction $2\Pi_2\prod_{r\mid d,\,r>2}\frac{r-1}{r-2}\,\mathrm{li}_2(\sqrt N)$ at $N=10^8$ (Row~1: $\mathrm{li}(\sqrt N)$), with $\mathrm{li}_2(x)=\int_2^x dt/\log^2t$ and $\Pi_2=0.6601618\ldots$.}\label{tab1}
\end{table}

\subsection{Discussion}
\begin{enumerate}\setlength\itemsep{1pt}
\item Among rows $i\ge2$, the maximum count occurs at Row~106 ($d=210$), followed by Row~211 ($d=420$): the odd part $105=3\cdot5\cdot7$ maximizes the singular-series factor $\prod\frac{r-1}{r-2}=3.2$ among $d\le490$. Row~4 ($d=6$) is maximal only within the initial rows $i\le5$. This ordering is exactly that predicted by Hardy--Littlewood \cite{HL}, and is the counting counterpart of Remark~\ref{singular}: rows with many small prime factors in $d$ are the richest in semiprimes.
\item The totals $12{,}667$, $30{,}222$, $74{,}418$ grow by the factor $74{,}418/30{,}222=2.46$, matching the prediction $2.42$ of the $N/\log^2N$ scaling.
\item The ratio $R_i^{\mathrm{semi}}(N)/\mathrm{HL}$ lies in $[0.86,1.03]$ across all 246 rows at $N=10^8$ (mean $0.93$) and increases slowly with $N$, consistent with the secondary terms of $\mathrm{li}_2$.
\end{enumerate}

\subsection{The minimal-anchor data}
Table~\ref{tab2} samples $\pmin(d)$ on windows of even $d$ starting at $D=10^4,\dots,10^{14}$.
\begin{table}[ht]\centering\small
\begin{tabular}{rrrr}\toprule
window start $D$ & hit rate $\pmin\le13$ & $10/\log D$ & median $\pmin$\\\midrule
$10^4$ & 0.732 & 1.086 & 7\\
$10^6$ & 0.580 & 0.724 & 13\\
$10^8$ & 0.461 & 0.543 & 17\\
$10^{10}$ & 0.356 & 0.434 & 26\\
$10^{12}$ & 0.307 & 0.362 & 31\\
$10^{14}$ & 0.242 & 0.310 & 43\\\bottomrule
\end{tabular}
\caption{Sampled even $d$ in windows starting at $D$ (2000--4000 values per window): the proportion with $\pmin(d)\le13$, the union-bound prediction $(\pi(13)-1)\cdot2/\log D=10/\log D$ of Theorem~\ref{fixed} (accurate once $10/\log D<1$; at small $D$ the union bound saturates), and the median of $\pmin(d)$. The declining hit rate illustrates the density-zero statement of Theorem~\ref{fixed}; the growing median illustrates Remark~\ref{singular} and Theorem~\ref{almostall}. The smallest even $d$ with $\pmin(d)>13$ is $d=112$, where $\pmin(112)=19$. The sampling script is provided as an ancillary file.}\label{tab2}
\end{table}

\section{Open problems}
\begin{enumerate}\setlength\itemsep{1pt}
\item[(1)] Prove $E_k>S_k-N_k$ for all $k\ge1$ (Legendre's conjecture reformulated via matrix $B$).
\item[(1$'$a)] Prove $\pmin(d)<\infty$ for all even $d$ (weak Polignac/Maillet; equivalently, every row of $B$ contains a semiprime). By Remark~\ref{deletion}, a proof must use information that distinguishes the primes from the balanced deletion sets described there.
\item[(1$'$b)] It is now known that $\frac12\log X\log\log X\le\max_{d\le X}\pmin(d)$ (Theorem~\ref{extremal}), while Conjecture~\ref{conj}(ii) predicts $\asymp(\log X)^2\log\log X$; determine the true order. Also: can the function $\psi$ in Theorem~\ref{almostall} be replaced by a constant, i.e.\ is $\pmin(d)\gg\log d\log\log d$ for almost all even $d$? The data of Remark~\ref{data} suggest the constant $\approx0.35$, which is out of reach of the first-moment method.
\item[(1$'$c)]\NEW{new in v6} Prove an upper bound of the form $\pmin(d)\le(\log d)^{C}$ for almost all even $d$, which together with Theorem~\ref{almostall} would determine the typical order of $\pmin$ up to a power of $\log d$.
\item[(1$'$d)]\NEW{new in v6} Prove that almost every row of $B$ contains \emph{infinitely many} semiprimes, i.e.\ that the rows with this property have density one. At present only a positive lower density is known (Corollary~\ref{posdens}, Theorem~\ref{blocks}); see Remark~\ref{levels} for the obstacle.
\item[(2)] Prove or disprove $E_k>S_k-N_k+1$, which would imply $P_k\ge2$.
\item[(3)] Establish $E_k-(S_k-N_k)\ge ck/\log k$ ($c>0$, all large $k$). By Theorem~\ref{exact} this asserts $P_k\gg k/\log k$, the expected order. The formulation $E_k\ge(S_k-N_k)(1+c/\log k)$ is false for large $k$: $S_k-N_k\gg k\log k$ unconditionally, while Brun--Titchmarsh gives $P_k\ll k/\log k$.
\item[(4)] Adapt the exact formula to half-intervals $((2k-1)^2,(2k)^2)$.
\item[(5)] Extend the formula to general intervals $(n^2,(n+1)^2)$.
\item[(6)] Identify $i_0$: determine which row $i_0\le124$ is guaranteed by Theorem~\ref{mtp} to contain infinitely many semiprimes. (Equivalent to identifying a specific even $d\le246$ occurring infinitely often as a prime difference.) More generally, make the constant $C$ of Theorem~\ref{blocks} effective, or reduce the number $123$ in Theorem~\ref{multiples}.
\item[(7)] Polignac via matrix $B$: for each $i\ge2$, prove Row $i$ has infinitely many semiprimes (= Polignac for gap $2(i-1)$, by Theorem~\ref{polignac}).
\item[(8)] Prime triples and adjacent rows: can infinitely many $(p,p+2,p+2i)$ be established and used, via Proposition~\ref{shared}, for simultaneous semiprime infinitude in adjacent rows?
\item[(9)] Goldbach via matrix $B$: use the anti-diagonal dictionary of Corollary~\ref{dict} to derive lower bounds for Goldbach representations. (Theorem~\ref{goldbach} is a first instance of the same counting applied to both problems.)
\item[(10)] Closed-form evaluation or meromorphic continuation of $F_i(s)$ for general $s$.
\end{enumerate}

\section{Conclusion}
We have developed both analytic and algebraic theories of matrix $B$. On the analytic side, the exact formula $P_k=N_k-S_k+E_k$ and its equivalent combinatorial condition $P_k\ge1\iff E_k>S_k-N_k$ give a precise account of primes in the intervals $((2k-1)^2,(2k+1)^2)$, verified unconditionally for $k\le999{,}999{,}999$ and, up to the threshold in \cite{BHP}, for $k\le5.49\times10^{11}$.

On the algebraic side, the natural injection $\varphi_i$ from Row $i+1$ into Row $i$ encodes a shared-factor structure: entries $b_{i+1,j}$ and $b_{i,j+1}$ always share the factor $2j+2i+1$, and prime triples simultaneously produce semiprimes in adjacent rows. The Dirichlet series $F_i(s)$ of Row $i$ satisfies the closed-form identity $F_i(1)=\frac{1}{2(i-1)}\sum_{k=1}^{i-1}\frac{1}{2k+1}$ ($i\ge2$). The transposition $(i,j)\mapsto(j,i)$ satisfies $b_{ij}+b_{ji}=v(v+3)$ along the anti-diagonal $v=2(i+j)-1$, and exhibits $B$ as a common dictionary for differences of primes (rows) and sums of primes (anti-diagonals).

For semiprime distributions: Row~1 trivially has infinitely many semiprimes; for $i\ge2$, infinitely many semiprimes in Row $i$ is equivalent to Polignac for gap $2(i-1)$; every row has infinitely many entries with $\Omega\le3$ by Chen; some row $i_0\le124$ has infinitely many semiprimes by Maynard--Tao--Polymath. The minimal-anchor function $\pmin$ (= OEIS A020483) organizes the row-coverage question: fixed anchors cover a density-zero set of rows (Theorem~\ref{fixed}); almost every row is covered, quantitatively, by Lavrik's theorem (Theorem~\ref{lavrik}); and the union bound, made uniform in the anchor bound, yields the unconditional lower bounds of Theorems~\ref{almostall} and \ref{extremal}, together with their Goldbach analogues (Theorem~\ref{goldbach}).

Between ``almost every row'' and ``every row'' lies an intermediate level that is a theorem: for every even $d$ one of the rows of gaps $d,2d,\dots,123d$, and one of any $C$ consecutive rows, contains infinitely many semiprimes (Section~\ref{sec68}). The gap that remains is exactly the weak Polignac conjecture. Section~\ref{sec67} shows that this gap cannot be closed by the structure of $B$ alone: the transfer identity, the transposition symmetry and the congruence obstructions hold equally for prime-like sets in which one prescribed row is empty. The matrix-$B$ framework thus offers a clean dictionary --- Polignac for gap $d$ $\iff$ Row $(d/2)+1$ contains infinitely many semiprimes; weak Polignac $\iff$ $\pmin(d)<\infty$ for all even $d$ --- and a precise statement of what remains to be proved, rather than a proof. We hope the row value identity, the shared-factor structure, the transposition symmetry and the function $\pmin$ will prove useful in the further study of these classical problems.

\medskip\noindent\textbf{Acknowledgments.} The author thanks the many students and colleagues who participated in discussions on matrix $B$ and its connections to open problems in prime number theory, and the Polymath project participants for their work on bounded gaps.

\medskip\noindent\textbf{Author contributions and AI usage statement.} The mathematical ideas, including the exact prime counting formula, the row injection structure, the algebraic framework for matrix $B$, the introduction of the minimal-anchor function $\pmin$, the anchor-by-anchor counting of Section~\ref{sec66}, and the questions on transposition, induction on rows and amplification of an empty row examined in Section~\ref{sec67}, are the work of the author. Manuscript preparation, computational implementations, numerical verification (including the data of Tables~\ref{tab1}--\ref{tab3} and of Remarks~\ref{singular} and \ref{deletion}), literature checking, and the analysis leading to the formulations of Sections~\ref{sec65}--\ref{sec67} (including the deletion-set constructions) were carried out with the assistance of AI tools under the author's direction; all mathematical statements and proofs were verified by the author. The verification scripts are provided as ancillary files.

\ifarxiv
\medskip\noindent\textbf{Version history.} v1: original submission. v2: corrected a false formulation in Open Problem (3). v3: Table~\ref{tab1} recomputed in full; Theorem~\ref{gap} re-proved via the prime-gap tables of \cite{OHP}; the crossover constant in Theorem~\ref{bhp} corrected. v4: adds Section~\ref{sec65} (the minimal-anchor function $\pmin$, the fixed-anchor density-zero theorem, Lavrik's almost-all theorem restated in the matrix-$B$ framework, and quantitative conjectures on $\pmin$), Table~\ref{tab2}, and open problems (1$'$a)--(1$'$b). v5: adds Section~\ref{sec66} (uniform union bound; unconditional lower bounds for $\pmin$ and for the least Goldbach prime), Proposition~\ref{anchorby}, Table~\ref{tab3}. v6 (this version): corrects the direction of the heuristic in Remark~\ref{singular} (large values of $\pmin$ occur at gaps \emph{coprime} to the small odd primes, not at gaps divisible by them; data added); adds Section~\ref{sec67} (row transfer, transposition symmetry and Goldbach dictionary, propagation of fences, deletion sets) and Section~\ref{sec68} (rows along multiples and in bounded blocks: Theorems~\ref{multiples}, \ref{blocks}, Corollary~\ref{posdens}); Table~\ref{tab2} moved to Section~7.3; open problems (1$'$c), (1$'$d) added; Remark~\ref{levels} distinguishes ``almost every row contains a semiprime'' (a theorem) from ``almost every row contains infinitely many'' (open); minor corrections (the size of $|A_q(X)|$ in the text before Table~\ref{tab3}; the display in the proof of Theorem~\ref{almostall}). All theorems of v5 are unchanged, and the numbering 6.1--6.25 is preserved.
\fi
\end{document}